\documentclass{amsart}[11pt]
\usepackage{graphicx, amssymb, amsfonts, enumerate, amscd, verbatim,
 color}

\newcommand{\map}{\phi}
\newcommand{\tth}{^{\operatorname{th}}}
\newcommand{\ZZ}{\mathbb{Z}}
\newcommand{\QQ}{\mathbb{Q}}
\newcommand{\PP}{\mathbb{P}}
\newcommand{\RR}{\mathbb{R}}

\newcommand{\CC}{\mathbb{C}}
\newcommand{\Aff}{\mathbb{A}}
\newcommand{\pp}{\mathfrak{p}}
\newcommand{\qq}{\mathfrak{q}}
\newcommand{\OO}{\mathcal{O}}

\newcommand{\pc}[1]{Y^{\operatorname{pre}}\left(#1\right)}
\newcommand{\cpc}[1]{X^{\operatorname{pre}}\left(#1\right)}

\renewcommand{\bar}[1]{#1\llap{$\overline{\phantom{\rm#1}}$}}

\newcommand{\col}{\,{:}\,}

\renewcommand{\setminus}{\smallsetminus}

\newtheorem{thm}{Theorem}[section]
\newtheorem*{thm*}{Theorem}
\newtheorem{lem}[thm]{Lemma}
\newtheorem{cor}[thm]{Corollary}
\newtheorem{prop}[thm]{Proposition}

\newtheorem*{conjecture*}{Conjecture}

\theoremstyle{remark}

\newtheorem{remark}[thm]{Remark}
\newtheorem{example}[thm]{Example}

\theoremstyle{definition}
\newtheorem{define}[thm]{Definition}

\begin{document}

\title[Uniform Bounds on Pre-Images]{Uniform Bounds on Pre-Images
under Quadratic Dynamical Systems}

\author[Faber]{Xander Faber}
\address{
Department of Mathematics and Statistics \\ 
McGill University \\
Montr\'eal, QC   H3A 2K6 \\ CANADA
}
\email{xander@math.columbia.edu}
\urladdr{http://www.math.columbia.edu/$\sim$xander/}

\author[Hutz]{Benjamin Hutz}
\address{
Department of Mathematics and Computer Science \\
Amherst College \\
Amherst, MA 01002 \\ USA
}
\email{bhutz@amherst.edu}

\author[Ingram]{Patrick Ingram}
\address{
Department of Pure Mathematics \\
University of Waterloo \\
Waterloo, ON  N2L 3G1 \\ CANADA
}
\email{pingram@math.uwaterloo.ca}

\author[Jones]{Rafe Jones}
\address{
Department of Mathematics \\
University of Wisconsin \\
Madison, WI  53706 \\ USA
}
\email{jones@math.wisc.edu}

\author[Manes]{Michelle Manes}
\address{
Department of Mathematics \\
University of Hawaii \\
Honolulu, HI  96822 \\ USA
}
\email{mmanes@math.hawaii.edu}

\author[Tucker]{Thomas J. Tucker}
\address{
Department of Mathematics \\
University of Rochester \\
Rochester, NY  14627 \\ USA
}
\email{ttucker@math.rochester.edu}

\author[Zieve]{Michael E. Zieve}
\address{
Department of Mathematics \\
Rutgers University \\
Piscataway, NJ 08854 \\ USA
}
\email{zieve@math.rutgers.edu}
\urladdr{http://www.math.rutgers.edu/$\sim$zieve/}


\subjclass[2000]{
14G05,
11G18
(primary);
37F10
(secondary)}

\begin{abstract}
For any elements $a,c$ of a number field $K$, let $\Gamma(a,c)$ denote
the backwards orbit of $a$ under the map $f_c\colon\CC\to\CC$ given by
$f_c(x)=x^2+c$.  We prove an
upper bound on the number of elements of $\Gamma(a,c)$ whose degree
over $K$ is at most some constant $B$.  This bound depends only on
$a$, $[K:\QQ]$, and $B$, and is valid for all $a$ outside an explicit
finite set.  We also show that, for any $N\ge 4$ and any $a\in K$
outside a finite set, there are only finitely many pairs $(y_0,c)\in\CC^2$
for which $[K(y_0,c)\col K]<2^{N-3}$
and the value of the $N\tth$ iterate of $f_c(x)$ at $x=y_0$ is $a$.
Moreover, the bound $2^{N-3}$ in this result is optimal.
\end{abstract}


\thanks{The authors thank the American Institute of Mathematics for supporting
this project in its initial stages.  The research of the third author was
partially supported by a PDF grant from NSERC of Canada.}

\maketitle


\section{Introduction}
\label{Introduction}


\subsection{Bounding the Number of Pre-Images}

For an elliptic curve $E$ over a number field $K$, the Mordell--Weil theorem
implies finiteness of the group $E_{\text{tors}}(K)$ of $K$-rational torsion
points on $E$.  Merel~\cite{Merel}, building on work of Mazur, Kamienny, and
others, proved that $\#E_{\text{tors}}(K)$ is bounded by a function of
$[K\col\QQ]$ (uniformly over all $K$ and $E$).  This implies the following
uniform bound on torsion points over extensions of $K$ of bounded
degree (see \cite[Cor.~6.64]{ads}):

\begin{thm} \label{Thm: Ell Curves}
Fix positive integers $B$ and $D$.  There is an integer $\lambda(B,D)$
such that for any number field $K$\,with
$[K\col\QQ] \leq D$, and for any elliptic curve $E/K$, we have
\[
\#\{P \in E(\bar K): [K(P)\col K]\le B \text{ and } [N]P = \OO
 \text{ for some $N \ge 1$}\} \le \lambda(B,D).
\]
\end{thm}

From a dynamical perspective, Theorem~\ref{Thm: Ell Curves} controls the number
of bounded-degree pre-images of the point $\OO$ under the various maps
$[N]\colon E\to E$.  In this paper we prove an analogue of this
result for maps $\Aff^1\to\Aff^1$ defined by the iterates of a degree-$2$
polynomial $f\in \bar\QQ[x]$.  Write $f^N$ for the $N\tth$ iterate of the
polynomial $f$.  A height argument similar to the one used by Mordell and Weil
shows that, for any number field $K$, any quadratic $f\in K[x]$, and any
$a\in K$ and $B>0$, the set
\[
\{x_0\in\bar K: [K(x_0)\col K]\le B \text{ and } f^N(x_0)=a
 \text{ for some } N\ge 1\}
\]
is finite.  The sizes of these sets cannot be bounded in terms of
$K$, $a$, and $B$: for any $N\ge 1$, put $f(x):=(x-b)^2+b$ where
$b:=a-2^{2^N}$, and note that $f^N(b+2)=a$.  However, we will prove
such a bound on these sets in case $f$ varies over the family of
polynomials
\[
f_c(x) := x^2 + c.
\]

\begin{thm} \label{Thm main intro}
Fix positive integers $B$ and $D$.  For all but finitely many values
$a\in\bar{\QQ}$, there is an integer $\kappa(B,D,a)$ with the following
property: for any number field $K$\,such that
$[K\col\QQ]\le D$ and $a\in K$, and for any $c\in K$, we have
\[
\#\left\{ x_0 \in \bar{\QQ} : [K(x_0)\col K] \leq B \text{ and }
 f_c^N(x_0) = a \text{ for some } N\ge 1 \right\} \leq \kappa(B,D,a).
\]
\end{thm}

Further, we give an explicit description of the excluded values $a$: they
are the critical values of the polynomials $f_c^j(0)\in \ZZ[c]$, for
$2\le j\le 4+\log_2(BD)$.  It follows that the number of such values is less
than $16BD$, and we will show that these values do not have the form
$\alpha/m$ with $\alpha$ an algebraic integer and $m$ an odd integer.  We do
not know whether the result would remain true if we did not exclude these
finitely many values $a$.  We prove that this is the case if $B=D=1$
(see Theorem~\ref{Thm: No Exceptional Points}).

We do not assert any uniformity in $a$ in Theorem~\ref{Thm main intro}, and
in fact such uniformity cannot hold (since $a$ can be chosen as $f_c^N(x_0)$
for fixed $c,N,x_0$).  Also, our proof gives no explicit bound on the constant
$\kappa(B, D, a)$, since we use a noneffective result due to Vojta (which
generalizes the Mordell conjecture).  Our proof of
Theorem~\ref{Thm main intro} carries over immediately
to the family of polynomials $g_c(x):=x^k+c$ for any fixed $k\ge 2$;
it would be interesting to analyze other families of polynomials.

In a different direction, if we fix $N$ and vary $c$, the choices of $B$
and $D$ become crucial:
\begin{thm} \label{Thm cool intro}
Let $K$ be a number field and fix $a\in K$ and $N\ge 4$.  There is a
finite extension $L$ of $K$ for which infinitely many pairs
$(y_0,c)\in\bar{K}\times\bar{K}$ satisfy $f_c^N(y_0)=a$ and
$[L(y_0,c)\col L]\le 2^{N-3}$.  Conversely, if $a$ is not a critical
value of $f_c^j(0)$ for any $2\le j\le N$, then only finitely many pairs
$(y_0,c)\in \bar{K}\times\bar{K}$ satisfy $f_c^N(y_0)=a$ and
$[K(y_0,c)\col K]<2^{N-3}$.
\end{thm}
In this result, some values $a$ must be excluded: for $a=-1/4$,
we will show that infinitely many pairs $(y_0,c)\in \bar\QQ\times\bar\QQ$
satisfy $f_c^N(y_0)=a$ and $[\QQ(y_0,c)\col \QQ]\le 2^{N-4}$.  Note that
$a=-1/4$ is the unique critical value of $f_c^2(0)=c^2+c=(c+1/2)^2-1/4$.
If we fix $c$ (and $N$ and $a$), then only finitely many $y_0\in\bar{\QQ}$
satisfy $f_c^N(y_0)=a$; thus the first part of Theorem~\ref{Thm cool intro}
would remain true if we required the occurring values of $c$ to be distinct.
We will discuss Theorem~\ref{Thm cool intro} further in the next subsection
after defining the analogues of modular curves for this problem.

A different dynamical analogue of Merel's result has been conjectured by
Morton and Silverman~\cite{morton_silverman}.
For a field $K$ and a non-constant endomorphism $\phi$ of a variety $V$
over $K$, define the set of preperiodic points for $\phi$ to be
\[
\operatorname{PrePer}(\phi) =
\{P \in V(\bar{K}): \,\phi^N(P)=\phi^M(P) \text{ for some }N>M\ge 0\}.
\]
In case $V$ is an elliptic curve and $\phi=[R]$ for some $R>1$,
the set $\operatorname{PrePer}(\phi)$ coincides with $V_{\text{tors}}(\bar K)$.
This motivates the following special case of the Morton--Silverman
conjecture:

\begin{conjecture*}
For any positive integer $D$, there is an integer $\mu(D)$ such that,
for all number fields $K$ of degree at most $D$ and all $c \in K$, we have
\[
\#(\operatorname{PrePer}(f_c) \cap \Aff^1(K)) \leq \mu(D).
\]
\end{conjecture*}

See \cite[\S3.3]{ads} for a discussion of this conjecture.


\subsection{Notation, Pre-Image Curves, and the Proof Strategy}

Let $K$ be a field whose characteristic is not~$2$.
For $c\in K$, view $f_c(x):=x^2+c$ as a mapping $\Aff^1_K \to \Aff^1_K$.
We will study the dynamics of this mapping, by which we mean the behavior
of points under repeated application of this map.

In order to prove results valid for all $c\in K$, it is convenient to first
treat $c$ as an indeterminate.  This will be our convention unless otherwise
specified.

\begin{define}
Fix an element $a \in K$ and a positive integer $N$.  We write $\pc{N,a}$ for
the algebraic set in $\Aff^2$ defined by $f_c^N(x)-a$.  If $\pc{N,a}$ is
geometrically irreducible (that is, irreducible over $\bar{K}$), we define the
\textbf{$N\tth$ pre-image curve} $\cpc{N,a}$ to be
the completion of the normalization of $\pc{N,a}$.
\end{define}


Note that a point $(x_0,c_0)\in\Aff^2(\bar{K})$ lies on $\pc{N,a}$ if and only
if $x_0$ is a pre-image of $a$ under the $N\tth$ iterate of the map
$x \mapsto f_{c_0}(x)$.  For example, since the map
$x\mapsto f_{a-a^2}(x)$ fixes $x=a$,
the point $(a,a-a^2)$ lies on $\pc{N,a}$ for every $N\ge 1$.  Likewise, since
$f_{-a^2-a-1}$ maps
\[
a\longmapsto -a-1\longmapsto a,
\]
for every $N\ge 1$ the points $(a,-a^2-a-1)$ and $(-a-1,-a^2-a-1)$ lie on
$\pc{2N,a}$ and $\pc{2N-1,a}$, respectively.

The following result gives a sufficient condition for irreducibility of
$\pc{N,a}$.

\begin{thm}\label{thm genus intro}
Suppose $N$ is a positive integer and $a \in K$ is not a critical value of
$f_c^j(0)$ for any $2\le j\le N$.  Then $\pc{N,a}$ is geometrically
irreducible, and the genus of $\cpc{N,a}$ is $(N-3)2^{N-2} + 1$.
\end{thm}

We now restate the main part of Theorem~\ref{Thm cool intro}:

\begin{cor} \label{Thm cool intro 2}
Let $K$ be a number field and fix $N\ge 4$ and $a\in K$ that is not a
critical value of $f_c^j(0)$ for any $2\le j\le N$.  Then only finitely many
$P\in\cpc{N,a}(\bar K)$ satisfy $[K(P)\col K]<2^{N-3}$, but there is a finite
extension $L$ of $K$ for which infinitely many $P\in\cpc{N,a}(\bar K)$
satisfy $[L(P)\col L]=2^{N-3}$.
\end{cor}

This result should be compared with a conjecture of Abramovich and Harris
\cite[p.~229]{abramovichharris}, which says that a curve $C$ over a number
field $K$ admits a rational map of degree at most $d$ to a curve
of genus $0$ or $1$ if and only if there is a finite extension $L$ of
$K$ for which infinitely many $P\in C(\bar \QQ)$ satisfy
$[L(P)\col L]\le d$.  In light of the above result, this conjecture
says that $2^{N-3}$ should be the minimal degree of any rational map from
$\cpc{N,a}$ to a curve of genus $0$ or $1$.  We will prove that this is in
fact the case (one minimal degree map is the composition
$\delta_4\circ\delta_5\circ\dots\circ\delta_N$, whose image is the
genus $1$ curve $\cpc{3,a}$, where the maps $\delta_M$ are defined below).
It should be noted, however, that Debarre and Fahlaoui have produced
counterexamples to the Abramovich--Harris conjecture
\cite[5.17]{debarrefahlaoui}.  Still, the conjecture is known to be true
when $d$ is small (due to Abramovich, Harris, Hindry, Silverman, and Vojta),
%
%
and it is important to understand when it holds.

Define a degree-$2$ morphism $\delta\colon \Aff^2\to\Aff^2$ by
$\delta(x,c)=(x^2+c,c)$.  For $N > 1$, let $\delta_N$ be the restriction of
$\delta$ to $\pc{N,a}$, so the image of $\delta_N$ is $\pc{N-1,a}$.
For any fixed $a\in K$, this gives a tower of algebraic sets and maps
\[
\cdots \stackrel{\delta_{N+1}}{\longrightarrow} \pc{N,a}
\stackrel{\delta_N}{\longrightarrow} \pc{N-1,a}
\stackrel{\delta_{N-1}}{\longrightarrow} \cdots
\stackrel{\delta_2}{\longrightarrow}\pc{1,a}.
\]
When $\pc{N,a}$ and $\pc{N-1,a}$ are geometrically irreducible, $\delta_N$
induces a degree-$2$ morphism $\delta_N\colon \cpc{N,a} \to \cpc{N-1,a}$.

Our strategy for proving Theorem~\ref{Thm main intro} in case $B=D=1$
is as follows: if $a\in\QQ$ is not a critical value of $f_c^j(0)$ for any
$j\in\{2,3,4\}$, then Theorem~\ref{thm genus intro} implies that
$\cpc{4,a}$ is a geometrically irreducible curve of genus $5$.
By the Mordell conjecture (Faltings' theorem \cite{faltings}),
$\cpc{4,a}(\QQ)$ is finite.  An argument involving heights shows
that any point in $\Aff^2(\QQ)$ has a total of finitely many pre-images in
$\Aff^2(\QQ)$ under the various iterates of $\delta$.
Thus the union of all $\pc{N,a}(\QQ)$ with $N\ge 4$ is finite.  To deduce
Theorem~\ref{Thm main intro} in case $B=D=1$, note that for each $N<4$
the number of points in $\pc{N,a}(\bar \QQ)$ having fixed values of $a$ and $c$
is at most $2^N$, and in particular is bounded independently of $c$.
The proof of Theorem~\ref{Thm main intro} for other values of $B$ and $D$
follows the same strategy, but instead of Faltings' theorem we use a
consequence of Vojta's inequality on arithmetic discriminants \cite{vojta};
this requires some additional arguments adapting Vojta's result to our
situation.

We remark that the algebraic sets $\pc{N,0}$ have arisen previously in the
context of the $p$-adic Mandelbrot set \cite{rafe}.  Also the sets
$\pc{2,a}$ occur implicitly in the study of uniform lower bounds on
canonical heights of morphisms \cite{ing}; we will discuss the connection
between such bounds and our results in Remark~\ref{rempat}.

The remainder of the paper is organized as follows.  In
\S\ref{Sec: Smooth and Irreducible} we give a criterion for nonsingularity of
$\pc{N,a}$ and prove that nonsingularity implies irreducibility.
In \S\ref{sec:genus}, in case $\pc{N,a}$ is nonsingular, we compute the genus
of $\cpc{N,a}$, as well as the minimal degree of any rational map from
$\cpc{N,a}$ to a curve of genus $0$ or $1$.  We then prove our arithmetic
results in \S\ref{Sec: Uniformity}.


\section{Smoothness and Irreducibility}
\label{Sec: Smooth and Irreducible}

In this section we determine when $\pc{N,a}$ is nonsingular, and we
show that $\pc{N,a}$ is irreducible whenever it is nonsingular.
Throughout this section, $K$ is an algebraically closed field whose
characteristic is not $2$.

\begin{prop} \label{smooth}
Fix a positive integer $N$.  For $a \in K$, the following assertions are
equivalent:
\begin{enumerate}[\textup(a\textup)]
\item $\pc{N,a}$ is nonsingular.
\item $\pc{M,a}$ is nonsingular for $1 \leq M \leq N$.
\item There do not exist an integer $j$ with $2 \leq j \leq N$ and an element
$c_0 \in K$ such that
\[
f_{c_0}^j(0) = a \quad\text{ and }\quad
\frac{\partial f_c^j(0)}{\partial c}\Big|_{c=c_0} = 0.
\]
\end{enumerate}
\end{prop}

\begin{remark}
Condition $(c)$ says that $a$ is not a critical value of $f_c^j(0)$ for
any $2\le j\le N$.
\end{remark}

\begin{proof}
It suffices to show that $(a)$ and $(c)$ are equivalent, since if $(c)$ holds
for some $N$ then it automatically holds for every smaller $N$.  In order to
prove equivalence of $(a)$ and $(c)$, we must describe the singular points on
$\pc{N,a}$.  A point $(x_0,c_0)\in \Aff^2(K)$ is a singular point on
$\pc{N,a}$ if and only if the following three equations are satisfied:
\begin{align}
f_{c_0}^N(x_0) & = a\label{x0precrit}\\
\frac{\partial f_{c_0}^N(x)}{\partial x}\Big|_{x=x_0} &=0\label{superat}\\
\frac{\partial f_{c}^N(x_0)}{\partial c}\Big|_{c=c_0} &=0.\label{centers}
\end{align}
By repeatedly applying the chain rule (and using that $f'_{c_0}(x) = 2x$), we
find
\begin{align*}
\frac{\partial f_{c_0}^N(x)}{\partial x}\Big|_{x=x_0} &=
f'_{c_0}\left( f^{N-1}_{c_0}(x_0) \right) \cdot
f'_{c_0}\left( f^{N-2}_{c_0}(x_0) \right) \cdot\cdots\cdot
f'_{c_0}\left( f_{c_0}(x_0) \right) \cdot f'_{c_0}\left(x_0 \right) \\
&= 2^N\prod_{i=0}^{N-1}f_{c_0}^i(x_0).
\end{align*}
Thus, equation~\eqref{superat} is equivalent to the existence of an integer $i$
with $0 \leq i \leq N-1$ such that $f_{c_0}^i(x_0) = 0$.  For any such $i$, we
have
\begin{align*}
\frac{\partial f_c^N(x_0)}{\partial c}\Big|_{c=c_0}
&= \frac{\partial \left(f_{c}^{N-i}\left(f_c^i(x_0)\right)\right)}
{\partial c}\Big|_{c=c_0} \\
&= \frac{\partial f_{c_0}^{N-i}(y)}{\partial y} \Big|_{y=0}\cdot
   \frac{\partial f_c^i(x_0)}{\partial c}\Big|_{c=c_0} +
   \frac{\partial f_c^{N-i}(0)}{\partial c}\Big|_{c=c_0}.
\end{align*}
Since $f_{c_0}^{N-i}(y)=f_{c_0}^{N-i-1}(y^2+c_0)$ is a polynomial in
$K[y^2]$, its partial derivative with respect to $y$ has zero constant
term, so
\[
\frac{\partial f_c^N(x_0)}{\partial c}\Big|_{c=c_0} =
\frac{\partial f_c^{N-i}(0)}{\partial c}\Big|_{c=c_0}.
\]
If $i=N-1$ then this common value is $\frac{\partial f_c(0)}{\partial c} = 1$,
which in particular is nonzero.  Thus, a point $(x_0,c_0)\in \Aff^2(K)$
is a singular point of $\pc{N,a}$ if and only if all three of the following are
satisfied:
\begin{align}
f_{c_0}^N(x_0) & = a\label{x0precrit_v2}\\
f_{c_0}^i(x_0) &= 0 \,\text{ for some $i$ satisfying $0 \leq i \leq N-2$}
\label{superat_v2}\\
\frac{\partial f_{c}^{N-i}(0)}{\partial c}\Big|_{c=c_0} &=0\label{centers_v2}.
\end{align}
When \eqref{superat_v2} holds, equation \eqref{x0precrit_v2} is equivalent to
\begin{equation}
\label{x0precrit_v3}
f_{c_0}^{N-i}(0) = a.
\end{equation}
Conversely, if $c_0$ and $i$ satisfy \eqref{centers_v2} and
\eqref{x0precrit_v3}, then there exists $x_0\in K$ satisfying
\eqref{superat_v2}.  This implies the equivalence of $(a)$ and $(c)$ (with
$j=N-i$).
\end{proof}

\begin{remark} \label{Rem: Singular count}
Assertion $(c)$ of Proposition~\ref{smooth} gives a criterion for checking
whether $\pc{N,a}$ is smooth.  In fact, it allows us to bound the number of
values $a \in K$ for which smoothness fails.  Namely, $(c)$ associates to
any such value $a\in K$ a pair $(j,c_0)$, where $2\le j\le N$ and $c_0$
is a root of $\frac{\partial f_c^j(0)}{\partial c}$.  Since this last
polynomial has degree $2^{j-1} - 1$, there are at most that many possibilities
for $c_0$ corresponding to a specified value $j$.  Summing over $2\le j\le N$,
we find that $\pc{N,a}$ is smooth for all but at most $2^N-N-1$ values
$a\in K$.  We checked that equality holds if $K$ has characteristic zero
and $N\le 6$, and we suspect equality holds in most situations.
For $2\le N\le 6$, there are precisely $2^{N-1}-1$ values $a\in\bar{\QQ}$ for
which $\pc{N,a}$ is singular but $\pc{N-1,a}$ is nonsingular, and in each case
these values $a$ are conjugate over $\QQ$.
\end{remark}

\begin{cor} \label{Cor: N=1}
The algebraic set $\pc{1,a}$ is nonsingular for any $a\in K$.
The algebraic set $\pc{2,a}$ is nonsingular for any
$a\in K\setminus\{-1/4\}$.
\end{cor}

\begin{prop} \label{irred}
For $a\in K$ and $N\ge 1$, if $\pc{N,a}$ is nonsingular then it is
irreducible.
\end{prop}

\begin{proof}
First note that $\pc{1,a}$ is irreducible for any $a \in K$, since the
defining polynomial $x^2 + c - a \in K[x,c]$ is linear in $c$.
Henceforth we assume $N > 1$.  If $\pc{N, a}$ is nonsingular,
then Proposition~\ref{smooth} implies $\pc{M,a}$ is also nonsingular for all
$M < N$.  We will show that, for $M-1 < N$, if $\pc{M-1,a}$ is irreducible,
then $\pc{M,a}$ is irreducible as well.  By induction, this implies $\pc{N, a}$
is irreducible.

Write the function field of $\pc{M-1,a}$ as $K(y,c)$,
where $f_c^{M-1}(y)=a$.  The function fields of the components of $\pc{M,a}$
are the extensions of $K(y,c)$ defined by the factors of $x^2+c-y$ in
$K(y,c)[x]$.  Since each such factor is monic in $x$, and has
coefficients in $K[y,c]$, the corresponding component contains a point
$(x_0,c_0)$ lying over any prescribed point $(y_0,c_0)$ of $\pc{M-1,a}$.
Choose $c_0 \in K$ satisfying $f_{c_0}^{M-1}(c_0) = a$, so
$(c_0, c_0)$ is a point of $\pc{M-1,a}$.  Then $(0,c_0)$ is the unique point
$P\in\pc{M,a}$ for which $\delta_M(P)=(c_0,c_0)$.  Thus $(0,c_0)$ is contained
in each component of $\pc{M,a}$, so since $\pc{M,a}$ is nonsingular it
must be irreducible.
\end{proof}

One can also prove this result geometrically: for the key step,
note that $\delta_M$ is a finite morphism, so if $\pc{M-1,a}$ is
irreducible then $\delta_M$ maps each component of $\pc{M,a}$ surjectively
onto $\pc{M-1,a}$.

\begin{remark}
In fact, $\pc{N,a}$ is typically irreducible even when it is singular.
For each $N\ge 1$, the previous two results imply irreducibility of $\pc{N,a}$
for all values $a \in K$ not on a short list of potential exceptions.
For $N\le 4$, we checked the values $a$ on these lists, and found that
$\pc{N,a}$ is irreducible for all $a \in K$ except $a=-1/4$.  On the
other hand, $\pc{N,-1/4}$ has two components for each $N$ with $2\le N\le 6$.
We suspect that larger values $N$ behave the same way.
\end{remark}


\section{Genus and gonality}
\label{sec:genus}

In this section, for all values of $N$ and $a$ for which $\pc{N,a}$ is
nonsingular, we compute the genus and gonality of $\cpc{N,a}$.  Recall that the
\textbf{gonality} is the minimum degree of a non-constant morphism
$\cpc{N,a}\to\PP^1$.  We also compute the minimum degree of a non-constant
morphism from $\cpc{N,a}$ to a curve of genus one.

Throughout this section, $K$ is an algebraically closed field whose
characteristic is not~$2$.

For a fixed value $a\in K$, we will compute the genus of $\cpc{N,a}$
inductively, by applying the Riemann-Hurwitz formula to the map
$\delta_N\colon \cpc{N,a} \to \cpc{N-1,a}$ defined in
Section~\ref{Introduction}.  We begin by computing the ramification of this
map.

\begin{lem} \label{ramification}
Pick $a\in K$ and $N\ge 2$ for which $\pc{N,a}$ is
nonsingular.  Then $f_c^N(0)=a$ for precisely $2^{N-1}$ values $c\in K$,
and the corresponding points $(0,c)\in \pc{N,a}(K)$ comprise all points
of $\cpc{N,a}(K)$ at which $\delta_N\colon \cpc{N,a}\to\cpc{N-1,a}$
ramifies.
\end{lem}

\begin{proof}
Since $\pc{N,a}$ is nonsingular, for each $1\le M\le N$ it follows that
$\pc{M,a}$ is nonsingular (by Proposition~\ref{smooth}) and hence
irreducible (by Proposition~\ref{irred}).

First consider $\delta_N$ on $\pc{N,a}$, which is defined by
$\delta_N(x,c) = (x^2+c, c)$.  The points with fewer than two
pre-images are the images of points with $x=0$, so $\delta_N$ ramifies at
precisely the points $(0,c)$ on $\pc{N,a}$.  For $c\in K$, the point
$(0,c)\in \Aff^2(K)$ lies on $\pc{N,a}$ if and only if $f_c^N(0)=a$.
Note that $f_c^N(0)-a$ is a polynomial in $K[c]$ of degree $2^{N-1}$.
If $c_0\in K$ is a repeated root of $f_c^N(0)-a$, then
\[
f_{c_0}^N(0)=a\quad\text{ and }\quad
\frac{\partial f_c^N(0)}{\partial c}\Big|_{c=c_0} = 0,
\]
contradicting our nonsingularity hypothesis (by Proposition~\ref{smooth}).
Thus $f_c^N(0)=a$ for precisely $2^{N-1}$ values $a\in K$, and
the corresponding points $(0,c)\in\pc{N,a}(K)$ comprise all points of
$\pc{N,a}(K)$ at which $\delta_N$ ramifies.

It remains to show that $\delta_N$ is unramified at the `cusps'
$\cpc{N,a}\setminus\pc{N,a}$.  Write the function field of $\cpc{M,a}$
as $K(x_M,c)$ where $x_M^2+c=x_{M-1}$ for $M>1$ and $x_1^2+c=a$.
At the infinite place $P_1$ of $K(x_1,c)$, the functions $x_1$ and $c$ have
poles of orders $1$ and $2$.
Inductively, assume $x_M$ and $c$ have poles of orders $1$ and $2$ at
a place $P$ of $K(x_M,c)$ which lies over $P_1$.
Then $y:=x_{M+1}/x_M$ satisfies
$y^2=(x_M-c)/x_M^2$, and since the right side has a nonzero finite
value at $P$, there are two possibilities for the value of $y$ at $P$.  Thus,
Kummer's theorem \cite[Thm.~III.3.7]{Stichtenoth} implies that $P$ lies
under two places of $K(x_{M+1},c)$, neither of which is ramified.
\end{proof}

\begin{thm}[Genus Formula]\label{genus}
Let $a\in K$, and let $N\ge 1$ be an integer for which $\pc{N,a}$ is
nonsingular.  Then $\cpc{N,a}$ is irreducible and has genus $(N-3)2^{N-2} + 1$.
\end{thm}

\begin{proof}
For each $M\le N$, the algebraic set $\pc{M,a}$ is nonsingular (by
Proposition~\ref{smooth}) and hence irreducible (by Proposition~\ref{irred}),
so also $\cpc{M,a}$ is irreducible.  All that remains is to calculate its
genus.

We proceed by induction on $N$.  Let $g(N)$ denote the genus of $\cpc{N,a}$.
Since $\pc{1,a}$ is defined by $x^2+c=a$, it is isomorphic to the $x$-line, so
$g(1)=0$ as desired.  Inductively, suppose $g(N-1)=(N-4)2^{N-3}+1$ for some
$N\ge 2$.  We compute $g(N)$ by applying the Riemann-Hurwitz formula to the
degree-$2$ morphism $\delta_N\colon\cpc{N,a}\to\cpc{N-1,a}$.
Lemma~\ref{ramification} shows that $\delta_N$ ramifies at precisely $2^{N-1}$
points, so
\begin{align*}
2g(N) - 2 &= 2\left[ 2g(N-1) - 2\right] +
	     \sum_{\substack{\text{ramified points} \\ \text{of $\cpc{N,a}$}}}1 \\
	  &= 2\left[ 2g(N-1) - 2\right] + 2^{N-1},\\
\intertext{whence}
g(N) &= 2g(N-1) - 1 + 2^{N-2} \\
     &= (N-4)2^{N-2} + 2 - 1 + 2^{N-2} \\
     &= (N-3)2^{N-2} + 1. \qedhere
\end{align*}
\end{proof}

\begin{example} \label{genusdata}
For a general choice of $a \in K$, we saw above that $\pc{N,a}$ is
irreducible and nonsingular.  Passing to the completed curves, the generic
picture looks like
\[
\cdots \stackrel{2-1}{\longrightarrow} \cpc{4,a}
 \stackrel{2-1}{\longrightarrow} \cpc{3,a}\stackrel{2-1}{\longrightarrow}
 \cpc{2,a} \stackrel{2-1}{\longrightarrow} \cpc{1,a}
\]
\[
\hspace{1.5cm} g(4) = 5 \hspace{1.1cm} g(3) = 1 \hspace{1cm} g(2) = 0
 \hspace{1cm} g(1) = 0 \hspace{0.4cm}
\]
The fact that $\cpc{4,a}$ has genus larger than~1 will be of arithmetic value
to us in the next section.

For later use, we also summarize the relevant behavior for small values of $N$
and those values of $a$ for which $\pc{N,a}$ is singular.  We used
Magma \cite{magma} to compute the data in the following table.

\renewcommand{\thetable}{\thethm}
\stepcounter{thm}
\begin{table}[pht]
\begin{tabular}{|c|c|c|c|}
\hline
	       & Algebraic & Irreducible &  \\
$a \in \bar{\QQ}$ & Set &    Components  & Genus \\
\hline
$a \in A_2$ & $\pc{2, -1/4}$ & $2$ & $0, 0$ \\
	    & $\pc{3, -1/4}$ & $2$ & $0, 0$ \\
	    & $\pc{4, -1/4}$ & $2$ & $1, 1$ \\
	    & $\pc{5, -1/4}$ & $2$ & $5, 5$ \\
\hline
$a \in A_3$ & $\pc{3, a}$ & $1$ & $0$ \\
	    & $\pc{4, a}$ & $1$ & $3$ \\
\hline
$a \in A_4$ & $\pc{4, a}$ & $1$ & $4$ \\
\hline
\end{tabular}

\caption{We denote by $A_N$ the set of values $a\in\overline\QQ$ for
which $\pc{N,a}$ is singular but $\pc{N-1,a}$ is nonsingular.  These sets may
be computed using the criterion in Proposition~\ref{smooth}.  For example,
$A_2 = \{-1/4\}$.  Also $\#A_3=3$ and $\#A_4=7$.  The last column gives the
genera of the irreducible components of the given algebraic set.}
\label{GenusTable}
\end{table}
\end{example}

\begin{remark}\label{remark special case}
The case $a=-1/4$ is of special interest for various reasons.  Here we note
that $\pc{4,-1/4}$ has infinitely many rational points (since each of its
components is the affine part of a rank-one elliptic curve over $\QQ$).  By
contrast, for any other value $a\in\QQ$, the above results imply that
$\pc{4,a}$ is an irreducible curve
of genus greater than one, and thus has only finitely many rational points by
the Mordell conjecture (Faltings' theorem \cite{faltings}).
\end{remark}


We now compute the gonality of $\cpc{N,a}$:
\begin{thm} \label{Thm: Gonality}
Let $a\in K$, and let $N\ge 2$ be an integer for which $\pc{N,a}$ is
nonsingular.  Then the gonality of $\cpc{N,a}$ is $2^{N-2}$.
\end{thm}

Our proof uses Castelnuovo's bound on the
genus of a curve on a split surface (see \cite[2.16]{kani} or
\cite[Thm.~III.10.3]{Stichtenoth}):

\begin{thm}
Let $C_1$, $C_2$, and $C$ be smooth, projective, geometrically integral curves
over $K$, and supose there is a generically injective map
$\psi\colon C\to C_1\times_K C_2$.  Let $g_i$ be the genus of $C_i$, let
$\pi_i$ denote projection from $C_1\times_K C_2$ onto its $i\tth$ factor, and
let $n_i$ be the degree of the map $\pi_i\circ\psi\colon C\to C_i$.  Then the
genus $g$ of $C$ satisfies
\[
g \leq n_1g_1 + n_2g_2 + (n_1-1)(n_2-1).
\]
\end{thm}

\begin{proof}[Proof of Theorem~\ref{Thm: Gonality}]
By Theorem~\ref{genus}, the curve $\cpc{2,a}$ has genus zero, so it is
isomorphic to $\PP^1$.  The composition
\[
\delta_N \circ \cdots \circ \delta_3\colon \cpc{N,a} \to \cpc{2,a} \cong \PP^1
\]
has degree~$2^{N-2}$, so the gonality of $\cpc{N,a}$ is at most $2^{N-2}$.  We
prove equality by induction on $N$.  Since this is clear for $N=2$, we may
assume that $\cpc{N-1,a}$ has gonality $2^{N-3}$.  Let
$\phi\colon \cpc{N,a} \to \PP^1$ be a non-constant morphism of
minimal degree.  If $\phi$ factors through the map $\delta_N$, then
$\deg \phi$ is twice the gonality of $\cpc{N-1,a}$, as desired.  So assume
$\phi$ does not factor through $\delta_N$.  Since $\delta_N$ has degree $2$,
it follows that the map
\[
(\delta_N,\phi)\colon\cpc{N,a}\to \cpc{N-1,a}\times \PP^1
\]
is generically injective, and now Castelnuovo's inequality implies that
\begin{align*}
g(N) &\leq 2g(N-1) + (2-1)(\deg \phi -1) \\
(N-3)2^{N-2} + 1 &\leq 2\left((N-4)2^{N-3}+1\right) + \deg \phi -1 \\
2^{N-2} &\leq \deg \phi.
\end{align*}
Thus the gonality of $\cpc{N,a}$ is $\deg \phi = 2^{N-2}$.
\end{proof}

\begin{cor}\label{Cor: Elliptic Gonality}
Let $a\in K$, and let $N\ge 3$ be an integer for which $\pc{N,a}$ is
nonsingular.  Then $2^{N-3}$ is the minimal degree of any nonconstant
morphism from $\cpc{N,a}$ to a genus one curve.
\end{cor}

\begin{proof}
Since the gonality of $\cpc{N,a}$ is $2^{N-2}$, and any genus one curve
admits a degree-$2$ map to $\PP^1$, any nonconstant morphism from $\cpc{N,a}$
to a genus-$1$ curve has degree at least $2^{N-3}$.  Conversely, this degree
occurs for the map
\[
\delta_N \circ \cdots \circ \delta_4\colon \cpc{N,a} \to \cpc{3,a}.\qedhere
\]
\end{proof}


\section{Arithmetic of pre-images}
\label{Sec: Uniformity}

Let $K$ be a number field.  For $a, c \in K$, we are interested in the size of
\[
\{ x_0 \in K : f_c^N(x_0) = a \text{ for some } N \ge 1 \},
\]
the set of pre-images of $a$ under iterates of $f_c$.  These sets can be
arbitrarily large if we allow
$a$ to vary (even if $c$ is fixed).  Indeed, if we choose $b \in K$ to be a
non-preperiodic point for $f_c$, and put $a = f_c^N(b)$, then the above set
contains (at least) the $N$ elements $b, f_c(b), \ldots, f_c^{N-1}(b)$.
In this section we show that the situation is different if we fix $a$ and
allow $c$ to vary.

In particular, we prove Theorem~\ref{Thm main intro}.  To illustrate the
method, we begin by proving the following special case (in which no
values $a$ need to be excluded):

\begin{thm} \label{Thm: No Exceptional Points}
Let $K$ be a number field, and pick $a \in K$.  There is an integer
$\nu(K, a)$ such that any $c \in K$ satisfies
\[
\# \left\{ x_0 \in K: f_c^N(x_0) = a \text{ for some } N \ge 1\right\} \le
\nu(K, a).
\]
\end{thm}

\begin{proof}
Suppose $M>0$ is chosen so that $\pc{M,a}(K)$ is finite.
For each $c\in K$, we must bound the union of the following two sets:
\begin{align*}
U_c &:= \{ x_0 \in K : f_c^N(x_0) = a \text{ for some } N < M \} \\
V_c &:= \{ x_0 \in K : f_c^N(x_0) = a \text{ for some } N \ge M \}.
\end{align*}
For fixed $c$ and $N$, the polynomial $f_c^N(z)$ has degree $2^N$, so
$\#U_c\le \sum_{N=1}^{M-1} 2^N = 2^M-2$.  If $V_c$ is nonempty, so
$f_c^N(x_0)=a$ for some $N\ge M$ and $x_0\in K$, then
$(f_c^{N-M}(x_0),c)\in\pc{M,a}(K)$.  Hence there are only finitely many
$c\in K$ for which $\#V_c>0$, and for each such $c$ the following lemma
shows that $V_c$ is finite.  Letting $S$ be the maximum value of $\#V_c$,
it follows that $\#(U_c\cup V_c)\le 2^M-2+S$.

It remains to prove that $\pc{M,a}(K)$ is finite for some $M$.
If $\pc{4,a}$ is nonsingular, then $\cpc{4,a}$ has genus~$5$ by
Theorem~\ref{genus}.  We apply the Mordell conjecture (Faltings' theorem)
to conclude that $\cpc{4,a}(K)$ is finite.  This implies that
$\pc{4,a}(K)$ is finite, so we may take $M=4$.
If $\pc{4,a}$ is singular and $a\ne -1/4$, then (as noted in
Table~\ref{GenusTable}) $\pc{4,a}$ is geometrically irreducible of genus
more than $1$, so again Faltings' theorem implies $\pc{4,a}(K)$ is
finite.  Finally, if $a=-1/4$ then (again from Table~\ref{GenusTable})
the set $\pc{5,a}$ has two geometrically irreducible components, both of
genus $5$, so again Faltings' theorem implies $\pc{5,a}(K)$ is finite.
Thus, for each $a\in K$, we have exhibited an integer $M$ for which
$\pc{M,a}(K)$ is finite, and the proof is complete.
\end{proof}

\begin{lem} \label{heights}
Let $a, c$ be elements of a number field $K$.  For any integer $B$, the set
\[
\left\{ x_0 \in \bar{\QQ} : [K(x_0)\col K] \leq B \text{ and }
                               f_c^N(x_0) =a \text{ for some } N\ge 1 \right\}
\]
is finite.
\end{lem}

\begin{proof}
We use standard properties of canonical heights of morphisms, which can
be found for instance in \cite[\S3.4]{ads}.  The canonical height function
$\hat{h}$ associated to $f_c$ satisfies the properties
\begin{gather*}
\hat{h}(z)\ge 0\\
\hat{h}(f_c(z))=2\hat{h}(z)\\
\hat{h}(z)=h(z)+O(1)
\end{gather*}
for all $z\in\bar\QQ$, where $h$ is the absolute logarithmic Weil height and
the implied constant depends only on $c$.

If $f^N_c(x_0)=a$ for some $N\ge 1$, then
\begin{align*}
h(x_0) = \hat{h}(x_0) + O(1) = 2^{-N} \hat{h}(a) + O(1) \leq \hat{h}(a) + O(1)
 = h(a) + O(1).
\end{align*}
In particular, the set described in the lemma is a collection of algebraic
numbers of bounded height and degree, and so is finite (for instance by
\cite[Thm.~3.7]{ads}).
\end{proof}

The proof of Theorem~\ref{Thm main intro} follows the same strategy as
that of Theorem~\ref{Thm: No Exceptional Points}, but instead of Faltings'
theorem we use a consequence of a more powerful theorem due to
Vojta.  We need some notation to state this consequence.

If $\map \colon C \to C'$ is a non-constant morphism of smooth projective
curves with ramification divisor $R_\phi$, define
\[
\rho(\map) = \frac{\deg R_\phi}{2\deg \map}.
\]

\begin{thm}[Song--Tucker--Vojta] \label{stthm}
If $\map \colon C\to C'$ is a non-constant morphism of smooth projective curves
defined over a number field $K$, then the set
\[
\Gamma( C, \map ) = \left\{ P \in C(\bar{\QQ}) : [K(P)\col K] < \rho(\map)
                                       \text{ and } K(\map(P)) = K(P) \right\}
\]
is finite.
\end{thm}

Vojta proved this result in case $C'=\PP^1$
(see \cite[Cor.~0.3]{vojta} and \cite[Thm.~A]{vojta2}),
as a consequence of a deep inequality on arithmetic discriminants.
Song and Tucker \cite[Prop.~2.3]{songtucker} generalized Vojta's proof to
deduce Theorem~\ref{stthm} for arbitrary $C'$.  Note that Theorem~\ref{stthm}
implies the Mordell conjecture:
if $C$ has genus at least $2$, then any non-constant morphism
$\map\colon C\to\PP^1$ satisfies $\rho(\map)>1$, so the finite set
$\Gamma(C,\map)$ includes $C(K)$.

\begin{remark}
We advise the reader of some typographical errors in \cite{songtucker}.
Specifically, the inequality $\ge$ in \cite[Cor.~2.1]{songtucker} should be a
strict inequality $>$, the displayed equality in \cite[Rem.~2.4]{songtucker}
should say $\deg R_f = (2g-2) - (2g'-2)\deg f$, and the inequality $>$ in
the next line should be $<$.
\end{remark}

We will apply Theorem~\ref{stthm} to composite maps of the form
$\delta_M\circ \delta_{M+1}\circ \dots \circ \delta_{M+J}$.  First we give a
consequence of Theorem~\ref{stthm} for arbitrary composite maps.

\begin{lem} \label{stlemma}
Let
\[
X_N\stackrel{\map_N}{\longrightarrow}X_{N-1}
\stackrel{\map_{N-1}}{\longrightarrow}\ \cdots\
\stackrel{\map_{3}}{\longrightarrow}  X_2
\stackrel{\map_{2}}{\longrightarrow}X_1\stackrel{\map_1}{\longrightarrow} X_0
\]
be a sequence of smooth projective curves defined over a number field $K$,
equipped with non-constant $K$-morphisms $\map_M\colon X_M\rightarrow X_{M-1}$
for each $1\leq M\leq N$, and put
\begin{align*}
B_N &:= \min_{ 1 \leq M \leq N} 2^{N-M} \rho(\map_M) \\
b_N &:= \min_{ 1 \leq M \leq N} \rho(\map_M).
\end{align*}
Then the set
\begin{equation} \label{finiteset}
\left\{P\in X_N(\bar{\QQ}):[K(P)\col K]< B_N\text{ and }
		[K(\map_1\circ\cdots \circ \map_N(P))\col K]\ge b_N\right\}
\end{equation}
is finite.
\end{lem}

\begin{proof}
By Theorem~\ref{stthm}, for each $M$ with $1\le M\le N$ the set
\[
\Gamma(M) := \left\{ P \in X_{M}(\bar{\QQ}) : [K(P)\col K] < \rho(\map_M)
                    \text{ and } K(P) = K(\map_M(P)) \right\}
\]
is finite.  For $1\le M\le N$, define $\psi_M\colon X_N\to X_{N-M}$ by
\[
\psi_M := \map_{N - M + 1} \circ \map_{N - M + 2} \circ \cdots \circ \map_N,
\]
and let $\psi_0$ be the identity on $X_N$.  Since $\psi_M$ is a finite
morphism,
\[
\Gamma := \bigcup_{M=0}^N \left\{ P \in X_N(\bar{\QQ}):
 \psi_{M}(P) \in \Gamma(N-M) \right\}
\]
is a finite union of finite sets, and so is finite.
We will show that if $P \in X_N(\bar{\QQ}) \setminus \Gamma$ satisfies
$[K(\psi_N(P))\col K] \ge b_N$ then $[K(P)\col K] \ge B_N$, which proves
that the set defined in \eqref{finiteset} is contained in the finite set
$\Gamma$.

Suppose $P \in X_N(\bar{\QQ}) \setminus \Gamma$ satisfies
$[K(\psi_N(P))\col K] \ge b_N$.  Then
\[
K\left( \psi_N(P) \right) \subset K\left( \psi_{N-1}(P) \right) \subset \cdots
\subset K\left( \psi_0(P) \right) = K(P).
\]
If we choose $j$ with $0\le j\le N-1$ and $\rho(\map_{N-j})=b_N$, then
\[
[K(\psi_j(P))\col K] \ge [K(\psi_N(P))\col K] \ge b_N = \rho(\map_{N-j}).
\]

Let $0 \leq J\leq N-1$ be the least integer such that
\[
[K(\psi_J(P))\col K] \ge \rho(\map_{N-J}).
\]
We may assume $J\ge 1$, since otherwise we obtain the desired conclusion
\[
[K(P)\col K] =[K(\psi_0(P))\col K]\ge \rho(\map_N) \ge B_N.
\]
By minimality, for $0\le j < J$ we have
\[
[K(\psi_j(P))\col K] < \rho(\map_{N-j});
\]
but $P\notin\Gamma$ implies $\psi_j(P)\notin\Gamma(N-j)$, so
\[
K(\psi_j(P)) \neq K(\psi_{j+1}(P)),
\]
and thus $[K(\psi_j(P))\col K(\psi_{j+1}(P))] \ge 2$.  It follows that
\begin{align*}
[K(P)\col K] &=
  \Bigl( \prod_{j=0}^{J-1} [K(\psi_{j}(P))\col K(\psi_{j+1}(P))] \Bigr)
    [K(\psi_J(P))\col K] \\
  &\ge 2^J \rho(\map_{N-J}) \ge B_N.
\end{align*}
This completes the proof that the finite set $\Gamma$ contains the set defined
in \eqref{finiteset}.
\end{proof}

We now prove Theorem~\ref{Thm cool intro}.

\begin{proof}[Proof of Theorem~\ref{Thm cool intro}]
Since the algebraic set $\pc{3,a}$ has a geometrically irreducible
component of genus $0$ or $1$, there is a finite extension $L$ of $K$
for which $\pc{3,a}(L)$ is infinite.  Since the composite map
$\psi:=\delta_4\circ\delta_5\circ\dots\circ\delta_N$ defines an endomorphism
of $\Aff^2$ of degree $2^{N-3}$, if $\psi(P)\in\pc{3,a}(L)$ then
$[L(P)\col L]\le 2^{N-3}$.  But $\psi(P)\in\pc{3,a}(\bar \QQ)$ if and only if
$P\in\pc{N,a}(\bar \QQ)$.  This proves the first part of
Theorem~\ref{Thm cool intro}.

Now suppose $a$ is not a critical value of $f_c^j(0)$ for any
$2\le j\le N$, so $\pc{M,a}$ is nonsingular for $M\le N$, whence
$\cpc{M,a}$ is defined.
Consider the tower of smooth projective curves
\[
\cpc{N, a}\stackrel{\delta_N}{\longrightarrow}\cpc{N-1, a}
\stackrel{\delta_{N-1}}{\longrightarrow}
\ \cdots\ \stackrel{\delta_2}{\longrightarrow}\cpc{1, a},
\]
where $\delta_M \colon \cpc{M, a} \to \cpc{M-1, a}$ is the usual map.  By
Lemma~\ref{ramification}, the degree of the ramification divisor of $\delta_M$
is $2^{M-1}$, so
$\rho(\delta_M) = 2^{M-3}$.  If we apply Lemma~\ref{stlemma} to this tower of
curves, we have (in the notation of that lemma) $B_N=2^{N-3}$ and $b_N=1/2$.
Theorem~\ref{Thm cool intro} follows.
\end{proof}

\begin{remark}
By Remark~\ref{remark special case}, the set
$\pc{4,-1/4}(\QQ)$ is infinite, so the above proof implies that
$\pc{N,-1/4}(\bar\QQ)$ contains infinitely many points of degree at
most $2^{N-4}$.  Thus, the critical value hypothesis in
Theorem~\ref{Thm cool intro} cannot be removed.
\end{remark}

The following refinement of Theorem~\ref{Thm main intro} is our main result:

\begin{thm}[Uniform Boundedness for Pre-Images] \label{Uniform Bounds}
Fix a positive integer $B$, and put $N=\lfloor 4+\log_2(B)\rfloor$.
For any $a\in\bar{\QQ}$ such that $\pc{N,a}$ is nonsingular, there is an
integer $\kappa(B, a)$ with the following property: for any $c \in \bar\QQ$,
we have
\[
\#\left\{ x_0 \in \bar{\QQ} : [\QQ(a,c,x_0)\col \QQ(a)] \leq B \text{ and }
 f_c^M(x_0) = a \text{ for some } M\ge 1 \right\} \leq \kappa(B,a).
\]
Moreover, $\pc{N,a}$ is singular for fewer than $16B$ values
$a\in\bar{\QQ}$.
\end{thm}


\begin{proof}
By Remark~\ref{Rem: Singular count}, there are at most $2^N-N-1$ values
$a\in\bar\QQ$ for which $\pc{N,a}$ is singular, which implies the final
statement.

Choose $a \in \bar{\QQ}$ such that $\pc{N,a}$ is nonsingular.
For any $c\in\bar{\QQ}$, the set described in the theorem is contained
in $U_c\cup V_c$, where
\begin{align*}
U_c &:= \{ x_0 \in \bar{\QQ} : f^M_c(x_0) = a \text{ for some } M<N \}, \\
V_c &:= \{ x_0 \in \bar{\QQ} : [\QQ(a,c,x_0)\col \QQ(a)] < 2^{N-3} \text{ and }
                             f^M_c(x_0) = a \text{ for some } M \ge N \}.
\end{align*}
By Theorem~\ref{Thm cool intro}, there are only finitely many points
$(y_0, c_0) \in \pc{N,a}(\bar{\QQ})$ for which
$[\QQ(a,y_0, c_0)\col  \QQ(a)] < 2^{N-3}$.
For each such $c_0$, Lemma~\ref{heights} implies $V_{c_0}$ is finite;
for any other $c$ we have $\#V_c=0$.  Letting $S$ be the maximum of
$\#V_c$ over all $c\in\bar\QQ$, it follows that $S$ is an integer
depending only on $N$ and $a$.  Since $f^M_c(z)$ has degree $2^M$, we have
$\# U_c< 2^N$, so $\#(U_c\cup V_c)<S+2^N$.
\end{proof}

Theorem~\ref{Uniform Bounds}, as well as several other results in this
paper, applies to values $a$ for which a particular $\pc{N,a}$ is
nonsingular.  We now describe a large class of such values $a$.

\begin{prop} \label{precritsmooth}
Let $\OO_K$ be the ring of integers in a number field $K$, and let $a \in K$.
Suppose $a$ is integral with respect to some prime ideal of $\OO_K$ lying over
$2$; in other words, $a = a_1 / a_2$ with $a_1, a_2 \in \OO_K$ and
$a_2 \not\in \pp$ for some $\pp \mid 2$.
Then $\pc{N,a}$ is nonsingular for every $N\ge 1$.
\end{prop}

\begin{proof}
By Proposition~\ref{smooth}, it suffices to show there do not exist an integer
$2\le j\le N$ and an element $c_0\in\bar{\QQ}$ for which
\[
f_{c_0}^j(0) = a \quad\text{ and }\quad
\frac{\partial f_c^j(0)}{\partial c}\Big|_{c=c_0} = 0.
\]
Suppose $j$ and $c_0$ satisfy these conditions, and write
$P(c)=f_c^j(0)-a\in K[c]$.  Letting $R$ be the localization of $\OO_K$ at the
prime ideal $\pp$, our hypothesis on $a$ shows that $P$ is a monic polynomial
over $R$.  Since $P(c_0)=0$, the ring $R[c_0]$ is integral over $R$, and so
contains a prime ideal $\qq$ lying above $\pp$.

Writing $P(c)=Q(c)^2+c-a$ with $Q=f_c^{j-1}(0)\in\ZZ[c]$, we have
$P'(c)=2Q(c)Q'(c)+1$.  By assumption, $c_0$ is a double root of $P(c)$, and so
\[
0 = P'(c_0) = 2Q(c_0)Q'(c_0) + 1.
\]
Since $Q(c_0)Q'(c_0)\in R[c_0]$, we may reduce this equation modulo $\qq$ to
obtain the contradiction
\[
0 \equiv 1 \pmod{\qq}.
\]
Thus $\pc{N,a}$ is nonsingular.
\end{proof}

In particular, this result applies to any algebraic integer $a$, or more
generally to any ratio $a=\alpha/m$ with $\alpha$ an algebraic integer
and $m$ an odd integer.  For such values $a$, we know the genus and
gonality of $\cpc{N,a}$, and moreover we have uniform bounds on the
pre-images of $a$ under the various maps $f_c$.

\begin{remark} \label{rempat}
Our results are related to the study of uniform lower bounds on
the canonical height $\hat{h}$ associated to $f_c$, as $c$ varies.
A special case of a conjecture of Silverman \cite[Conj.~4.98]{ads}
asserts that, for every number field $K$, there exists a constant
$\epsilon=\epsilon(K)>0$ such
that either $\hat{h}(\alpha)=0$ or $\hat{h}(\alpha)\geq \epsilon
\max(1,h(c))$ for
each $\alpha, c\in K$.  (This is a dynamical analogue of a conjecture of
Lang's on heights of non-torsion rational points on elliptic curves.)
If this conjecture were true, we could prove
Theorem~\ref{Thm: No Exceptional Points} without using Faltings' theorem,
so long as we assume that $a$ is not preperiodic for $f_c$.  For such
$a$ and $c$, if $f_c^N(x_0)=a$ then $x_0$ is not preperiodic for $f_c$,
so $\hat h(x_0)\ne 0$ and thus
\[
2^N \epsilon \max(1,h(c))\leq 2^N\hat{h}(x_0)=\hat{h}(a)\leq h(a)+
h(c)+\log 2,
\]
where the last inequality follows from decomposing the heights into
sums of local heights.
This bounds $N$ in terms of $K$, $h(a)$, and $\epsilon$;
the rest of the proof is as before.  Partial results in the direction of
Silverman's conjecture (see \cite{ing}) imply an effective version of
Theorem~\ref{Thm main intro} if the bound $\kappa$ is allowed to depend on the
number of primes of $K$ at which $c$ is not integral (in addition to $BD$ and
$a$).  Of course, this is much weaker than Theorem~\ref{Thm main intro},
in which $\kappa$ does not depend on $c$.

In the other direction, since $\cpc{3,0}$ is a rank-one elliptic
curve over $\QQ$, with unbounded real locus, there are infinitely many
$(x_0,c)\in\pc{3,0}(\QQ)$
with $|c|>4$.
%
%
For such $(x_0,c)$ we have $f_c^4(x_0)=f_c(0)=c$, so
\cite[Lemmas~3 and 6]{ing} imply
\[
\hat{h}(x_0)=2^{-4}\hat{h}(c)\leq
\frac{1}{16}h(c)+\frac{\log(5)-2\log(2)}{16}.
\]
Thus, if $\epsilon(\QQ)$ exists then it is at most $1/16$.
A similar construction was given in \cite[\S 5]{ing}, using the
points $(k,-k^2-k+1)$ on $\pc{2,-3k+2}$ to deduce an upper bound
of $1/8$; note that that construction exhibits an infinite family of integral
points, whereas each curve $\cpc{2,a}$ has only finitely many such
points (since it is a genus zero curve with two rational points at
infinity).
%
%
\end{remark}


\section*{Acknowledgements}

This project began at the American Institute of
Mathematics, during the workshop on ``The Uniform Boundedness Conjecture in
Arithmetic Dynamics''.  We thank AIM for the opportunity to enjoy its
productive
atmosphere.  We also thank the workshop organizers and participants for
contributing to this stimulating week.  We especially thank the participants
with whom we discussed this work:
Jordan Ellenberg, Susan Goldstine, Bjorn Poonen, Joseph Silverman,
Vijay Sookdeo, Michael Stoll, and Justin Sukiennek.


\end{document}
